\documentclass{amsproc}
\usepackage{eurosym}
\usepackage{amssymb}
\usepackage{amsfonts}
\usepackage{cmtt}

\usepackage[utf8]{inputenc}
\usepackage[T1]{fontenc}

\theoremstyle{plain}
\newtheorem{theorem}{Theorem}[section]

\newtheorem{question}[theorem]{Question}

\newtheorem{proposition}[theorem]{Proposition}
\newtheorem{lemma}[theorem]{Lemma}
\newtheorem{remark}[theorem]{Remark}
\theoremstyle{remark}
\newtheorem*{proofofmain}{Proof of Theorem \ref{TheoremMain}}

\newcommand{\field}[1]{\mathbb{#1}}

\newcommand{\N}{\field{N}}

\def\kwadrat{\hfill$\square$}

\begin{document}
	
\title{Optimal stopping for many connected components in a graph}

\author[Micha\l\ Laso\'{n}]{Micha\l\ Laso\'{n}}

\dedicatory{\upshape 
Institute of Mathematics of the Polish Academy of Sciences,\\ ul.\'{S}niadeckich 8, 00-656 Warszawa, Poland\\ \textmtt{michalason@gmail.com}}


\begin{abstract}
We study a new optimal stopping problem: 
Let $G$ be a fixed graph with $n$ vertices which become active on-line in time, one by another, in a random order. The active part of $G$ is the subgraph induced by the active vertices. Find a stopping algorithm that maximizes the expected number of connected components of the active part of $G$.

We prove that if $G$ is a $k$-tree, then there is no asymptotically better algorithm than `wait until $\frac{1}{k+1}$ fraction of vertices'. The maximum expected number of connected components equals to
$$\left(\frac{k^k}{(k+1)^{k+1}}+o(1)\right)n.$$
\end{abstract}
	
\maketitle

\section{Introduction}

Optimal stopping is an area of mathematics gathering problems of choosing a right time in order to maximize an expected reward. The difficulty of these problems comes from the fact that the algorithm can choose only the currently available reward. The algorithm has to make a decision knowing only the current reward and rewards that were possible in the past, while about the future rewards it knows only the distribution. 

Optimal stopping problems emerged as direct abstractions of real-life questions, that is why they can be found in several areas of applied mathematics, statistics, economics, and mathematical finance.

Probably the most famous optimal stopping problem is the `secretary problem', see \cite{Li61,GiMo66,Fe89}. The problem has many generalizations into other combinatorial objects: orders \cite{Gn92,Mo98,Pr99,GeKuMoNi08,Ko10,FrWa10}, direct graphs \cite{Su10,KuMo05,GoKuKu13,GrMoSu15,BeSu17}, and others.

\smallskip 

Morayne and Sulkowska proposed a new natural optimal stopping problem: 
\smallskip
\newline
\emph{Vertices of an unlabelled path of length $n$ appear on-line, one by another, in a random order. At a time only the induced subgraph on vertices that came already is visible. Find a stopping algorithm that maximizes the expected number of connected components of the visible subgraph.}

\smallskip 

In this paper we consider a slight generalization of Morayne--Sulkowska model. We introduce three variants of the problem on a graph with growing amount of information a stopping algorithm gets during the game -- starting from no information (blindness), via Morayne--Sulkowska model, up to full information. As it will turn out, on $k$-trees asymptotically there is no difference between the scores in all three variants. Here we introduce the problem in full detail and necessary notions.

\smallskip

\textbf{Setting:} Let $G$ be a fixed graph with $n$ vertices which become active on-line in time, one by another, in a random order (permutation) $\sigma\in S_n$. An edge of $G$ is active if both its endpoints are active, thus the active part of $G$ is the subgraph induced by the active vertices. Denote by $CC(\sigma,t)$ the number of connected components in the active part of $G$ at a time $t$ on a permutation $\sigma$ -- that is, the number of connected components in the induced subgraph $G[\{\sigma(1),\dots,\sigma(t)\}]$.

We consider stopping algorithms -- algorithms $\mathcal{A}$ that know $G$ in advance, on a permutation $\sigma$ get some information about its active part $\mathfrak{I}(\sigma(1),\dots,\sigma(t))$ on-line in $t$, and  choose a stopping time $\mathcal{A}(\sigma)$. Equivalently, a stopping time is a function $\mathcal{A}$ such that: if $\mathfrak{I}(\sigma(1),\dots,\sigma(\mathcal{A}(\sigma)))=\mathfrak{I}(\tau(1),\dots,\tau(\mathcal{A}(\sigma)))$, then $\mathcal{A}(\tau)=\mathcal{A}(\sigma)$.

\smallskip

\textbf{Problem:} Find a stopping algorithm $\mathcal{A}$ that maximizes the expected value of the number of connected components, that is $CC(\mathcal{A})=\frac{1}{n!}\sum_{\sigma\in S_n}CC(\sigma,\mathcal{A}(\sigma))$. Find this maximum expected value $\max_{\mathcal{A}}CC(\mathcal{A})$. 

\smallskip

\textbf{Variants:} What information about the active part of $G$ algorithm gets on-line:
\begin{enumerate}
	\item[(B)] Blind -- during the game the algorithm knows only the number of vertices that came already. In particular, the algorithm does not know which vertices are active, or how many connected components are there, etc.
	\item[(PI)] Partial information -- during the game the algorithm has some partial information about the active part. E.g. if it is the current active part up to isomorphism, then it coincides with Morayne--Sulkowska model.
	\item[(FI)] Full information -- during the game the algorithm knows which vertices are active, so it knows everything that is possible. Of course, the algorithm does not know the order of vertices that are about to come. 
\end{enumerate}

\smallskip

\textbf{Results:} In Section \ref{Tree} we prove, within the realm of graphs, that for a tree $G$:
\begin{itemize}
	\item[(B)] in the Blind variant (Theorem \ref{TheoremBlind}): \newline the optimal algorithm $\mathcal{A}$ is to `wait until half of vertices', for which \newline $CC(\mathcal{A})=\frac{1}{4}n\pm 1$,
	\smallskip
	\item[(FI)] in the Full information variant  (Theorem \ref{TheoremMain}): \newline $\max_{\mathcal{A}}CC(\mathcal{A})=\left(\frac{1}{4}+o(1)\right)n$.
\end{itemize}

\smallskip

A natural generalization of trees are $k$-trees. Recall that a graph $G$ is a $k$-tree if it is maximal (w.r.t. the inclusion of edges) graph with treewidth equal to $k$. In Section \ref{kTree} we generalize the above results to $k$-trees. In particular, we reprove the case of trees, however using a more abstract approach. If $G$ is a $k$-tree, then:
\begin{itemize}
	\item[(B)] in the Blind variant (Theorem \ref{kTheoremBlind}): \newline an almost optimal algorithm $\mathcal{A}$ is to `wait until $\frac{1}{k+1}$ fraction of vertices', $CC(\mathcal{A})=\frac{k^k}{(k+1)^{k+1}}n\pm\frac{k+2}{e}$,
	\smallskip
	\item[(FI)] in the Full information variant (Theorem \ref{kTheoremMain}): \newline $\max_{\mathcal{A}}CC(\mathcal{A})=\left(\frac{k^k}{(k+1)^{k+1}}+o(1)\right)n$, \newline this upper bound holds also for all maximal $k$-degenerate graphs.
	
\end{itemize}

\smallskip

To conclude -- in all three variants of the optimal stopping problem maximizing the expected number of connected components (including Morayne--Sulkowska model) on a $k$-tree there is no asymptotically better algorithm than simply `wait until $\frac{1}{k+1}$ fraction of vertices'. This gives the maximum expected number of connected components 
$$\left(\frac{k^k}{(k+1)^{k+1}}+o(1)\right)n.$$ 
Surprisingly, for $k$-trees when only asymptotics matter, full information does not give any advantage compared to blindness. However, this is not the case for all graphs -- see Remark \ref{RemarkFI>B}.

Firstly, in Section \ref{ArbitraryGraph} we show a general framework for an optimal algorithm on a graph in the full information variant. 

\section{Arbitrary graphs}\label{ArbitraryGraph}

\subsection*{Full information variant}

It is reasonable not to stop whenever the expected gain of the next move is nonnegative. 

\begin{proposition}\label{Proposition}
Let $G$ be an arbitrary graph. Suppose in $l$-th step $G[l]$ has exactly $c$ connected components $C_1,C_2,\dots,C_c$. Denote by $N(C_i)$ the neighborhood in $G$ of the component $C_i$ . An optimal strategy does not stop whenever an inequality $n-l\geq \lvert N(C_1)\rvert+\dots+\lvert N(C_c)\rvert$ holds.
\end{proposition}

\begin{proof}
	Let $D_j$ denote the set of non active vertices that are in exactly $j$ among neighborhoods $N(C_i)$, for $j=1,\dots,c$, and let $d_j$ be their cardinalities. Then, the number $c'$ of connected components after $(l+1)$-th step increases by one $c'=c+1$ if $(l+1)$-th vertex is from the set of non active vertices and it does not belong to any neighborhood $N(C_i)$, that is, if it is from the set of cardinality
    \[
    \begin{split}
    & \;\lvert(V\setminus (C_1\cup\dots\cup C_c))\setminus (N(C_1)\cup\dots\cup N(C_c))\rvert\\
	=  &  \;\vert V\setminus (C_1\cup\dots\cup C_c)\rvert-\lvert N(C_1)\cup\dots\cup N(C_c)\rvert\\
	=  &  \;n-l-\lvert D_1\rvert-\dots-\lvert D_c\rvert=n-l-d_1-\dots-d_c.
	\end{split}
	\]
	The number $c'$ of connected components after $(l+1)$-th step decreases exactly by $j-1$ if $(l+1)$-th vertex is from the set $D_j$. Thus, the expected number of connected components after $(l+1)$-th step equals to
	\[
	\begin{split}
	\mathbb{E}(c') & =c+\frac{1}{n-l}\left((n-l-d_1-\dots-d_c)-d_2-2d_3\dots-(c-1)d_c\right)\\
	& =c+\frac{1}{n-l}\left((n-l-d_1-2d_2-\dots-cd_c\right).
	\end{split}
	\]
	Thus,
	$$\mathbb{E}(c'-c)=\frac{1}{n-l}\left(n-l-(\lvert N(C_1)\rvert+\dots+\lvert N(C_c)\rvert)\right),$$ 
	so it is profitable to make the next step if $n-l\geq\lvert N(C_1)\rvert+\dots+\lvert N(C_c)\rvert$. 
\end{proof}

It is tempting to guess that the above inequality describes an optimal algorithm. However, already on a tree it may happen that the expected gain of the next move is negative, but even though it is profitable to continue. This gives a warning that an exact description of an optimal algorithm could be out of reach.

\begin{remark}
Sometimes it is profitable to continue, even though the expected gain of the next step is negative -- that is, when the inequality from Proposition \ref{Proposition} is opposite: $n-l< \lvert N(C_1)\rvert+\dots+\lvert N(C_c)\rvert$.
\end{remark}

Consider a star with $n+1$ leafs and one path of length $n-1$ attached. Suppose that exactly all leafs of the star are active. Then, the number of vertices which are about to come is $n$, and the sum of sizes of neighborhoods is $n+1$. However, still it is profitable to continue -- consider the following continuation strategy: if the next vertex is not the center of the star then `stop', otherwise (when the center of the star came) `take $\frac{n-1}{2}$ more vertices'.
The expected value of the continuation strategy is
$$\frac{n-1}{n}\cdot 1+\frac{1}{n}\left(-(n+1)+\frac{1}{4}(n-1)\right)\simeq\frac{3}{4}>0.$$

\section{Trees}\label{Tree}

\subsection*{Blind variant}\label{Blind}

\begin{theorem}\label{TheoremBlind}
	Suppose $G$ is a tree. An optimal strategy is to stop after $l=\lfloor\frac{n+1}{2}\rfloor$ or $\lceil\frac{n+1}{2}\rceil$ vertices. The maximum expected number of connected components satisfies 
	$$\frac{1}{4}n<\max_{\mathcal{A}}CC(\mathcal{A})<\frac{1}{4}n+1.$$
\end{theorem}

\begin{proof}
	Suppose the algorithm stops after $l(n)$ steps. Notice that the stopping function $l$ depends only on $n$, since the algorithm doesn't have any other information. 
	
	For any forest $F$ we have $CC(F)=\lvert V(F)\rvert-\lvert E(F)\rvert$. Thus, for forests $G[l]$ we have $\mathbb{E}(CC[l])=\mathbb{E}(\lvert V[l]\rvert)-\mathbb{E}(\lvert E[l]\rvert)$. Clearly, $\lvert V[l]\rvert=l$, so $\mathbb{E}(\lvert V[l]\rvert)=l$.
	\[
	\begin{split}
	\mathbb{E}(\lvert E[l]\rvert) & =\sum_{e\in E(G)}\mathbb{P}(e\in G[l]) = \sum_{vw\in E(G)}\mathbb{P}(v,w\in G[l]) \\
	& = \sum_{vw\in E(G)}\frac{l(l-1)}{n(n-1)} = \frac{l(l-1)}{n}.
	\end{split}
	\]
	Hence, 
	$$\mathbb{E}(CC[l])=l-\frac{l(l-1)}{n}=\frac{l(n-l+1)}{n}.$$
	It maximizes for integer $l=\lfloor\frac{n+1}{2}\rfloor$ or $\lceil\frac{n+1}{2}\rceil$, and then $\frac{n}{4}<\mathbb{E}(\vert CC[l]\rvert)<\frac{n}{4}+1$. 
\end{proof}

\subsection*{Full information variant}

We prove that asymptotically there is no better algorithm than `wait until half of vertices'. Thus, the best result is $\left(\frac{1}{4}+o(1)\right)n$ connected components. 
In particular, when only asymptotic constant matters full information does not give any advantage to the algorithm compared to blindness.

\begin{theorem}\label{TheoremMain}
For every $\varepsilon>0$ there exists an integer $N_{\varepsilon}$ such that if $G$ is a tree with $n\geq N_{\varepsilon}$ vertices, then 
$$\max_{\mathcal{A}}CC(\mathcal{A})\leq\left(\frac{1}{4}+\varepsilon\right)n.$$
\end{theorem}

We recall Janson's inequality, which will be used in the proof. 

\begin{theorem}[Janson \cite{Ja90}]\label{TheoremJanson}
Suppose a subset $R$ is drawn randomly from a finite set $V$ such that the inclusions of individual elements are independent. Let $\mathcal{A}$ be a family of subsets of $V$. For every set $A\in\mathcal{A}$ consider a random variable $X_A$ indicating whether $A\subset R$. Denote $X=\sum_{A\in\mathcal{A}}X_A$. Then, for every $0<\varepsilon\leq 1$
$$\mathbb{P}(X\leq(1-\varepsilon)\mathbb{E}X)\leq exp\left(-\frac{1}{2}\frac{(\varepsilon\mathbb{E}X)^2}{\mathbb{E}X+\sum_{(A,B):A\neq B,A\cap B\neq\emptyset}\mathbb{E}X_AX_B}\right).$$
\end{theorem}

First we prove a concentration lemma for forests with not too large maximum degree. Notice that this assumption is essential. For a star the assertion of the lemma is false.

\begin{lemma}\label{LemmaConcentration}
For every $\varepsilon>0$ there exists $\delta_{\varepsilon}>0$ and an integer $M_{\varepsilon}$ such that if $G$ is a forest on $n\geq M_{\varepsilon}$ vertices with $\beta n$ edges and maximum degree at most $\delta_{\varepsilon} n$, then for every $\alpha\in[0,1]$ the probability that
$$CC(G[\alpha n])>(\alpha-\alpha^2\beta)n+\frac{3\varepsilon}{10}n$$
is at most $\frac{\varepsilon^3}{2000}$.
\end{lemma}

\begin{proof}
Let $G=(V,E)$ be a forest with $n$ vertices and $\beta n$ edges. When $\alpha<\frac{3\varepsilon}{10}$, or $\beta n<\frac{\varepsilon}{10}n$ the assertion is clear. Thus, $\alpha\geq\frac{3\varepsilon}{10}$ and $\beta n\geq\frac{\varepsilon}{10}n$.

We choose from $V$ a subset $S$ of $\alpha n$ vertices at random, i.e. such that every set of $\alpha n$ vertices is equally probable. We want to argue that the number of connected components of $G[S]$ exceeds its expected value by a positive fraction only with a small probability. 

We can make such a random choice of $S$ by a following procedure:
\begin{enumerate} 
\item drawn subset $R$ of $V$ randomly, such that each vertex is taken independently with probability $\alpha-\frac{\varepsilon}{10}$,
\item then $\mathbb{E}|R|=(\alpha-\frac{\varepsilon}{10})n$ and already from the binomial distribution it follows that the probability $\mathbb{P}\left(|R|\in\left((\alpha-2\frac{\varepsilon}{10})n,\alpha n\right)\right)$ grows to $1$ as $n$ tends to infinity, in particular for sufficiently large $n$ this probability is at least $\frac{1}{2}$,
\item if the size of $R$ is not from the interval $\left((\alpha-2\frac{\varepsilon}{10})n,\alpha n\right)$ resample $R$,
\item if the size of $R$ is from the interval $\left((\alpha-2\frac{\varepsilon}{10})n,\alpha n\right)$, add $\alpha n-|R|<2\frac{\varepsilon}{10}n$ vertices at random (such that every set is equally probable) and return $S$.  
\end{enumerate}
Inclusions of individual elements in $R$ are independent, so we may use Janson's inequality (Theorem \ref{TheoremJanson}) for $R$ and $\mathcal{A}=E$. Then, $X$ is the number of edges in $R$, $\mathbb{E}X=\beta n(\alpha-\frac{\varepsilon}{10})^2$, and
\[
\begin{split}
\mathbb{P}\left(X\leq\left(1-\frac{\varepsilon}{10}\right)\mathbb{E}X\right) &\leq exp\left(-\frac{1}{2}\frac{(\frac{\varepsilon}{10}\beta n(\alpha-\frac{\varepsilon}{10})^2)^2}{\beta n(\alpha-\frac{\varepsilon}{10})^2+2\beta n\delta_{\varepsilon}n}\right)=\\
=exp\left(-\frac{1}{2}\frac{(\frac{\varepsilon}{10})^2(\alpha-\frac{\varepsilon}{10})^4\beta n}{(\alpha-\frac{\varepsilon}{10})^2+2\delta_{\varepsilon}n}\right) &\leq exp\left(-\frac{1}{2}\frac{(\frac{\varepsilon}{10})^2(\alpha-\frac{\varepsilon}{10})^4\frac{\varepsilon}{10}n}{(\alpha-\frac{\varepsilon}{10})^2+2\delta_{\varepsilon}n}\right),
\end{split}
\]
so for sufficiently large $n$ and small $\delta_{\varepsilon}$ this probability is smaller than $\frac{\varepsilon^3}{4000}$. Hence, with probability greater than $1-\frac{\varepsilon^3}{2000}$ (the probability of failure is doubled because of point $(2)$) the number of connected components of $G[S]$ is at most 
\[
\begin{split}
CC(G[S]) & =|V(S)|-|E(S)|\leq |V(S)|-|E(R)|=\alpha n-\vert X\vert\\
& \leq\alpha n-\left(1-\frac{\varepsilon}{10}\right)\mathbb{E}X=\alpha n-\left(1-\frac{\varepsilon}{10}\right)\beta n\left(\alpha-\frac{\varepsilon}{10}\right)^2\\
& \leq(\alpha-\alpha^2\beta)n+\frac{3\varepsilon}{10}n,
\end{split}
\]
or, in other words, the number of components exceeds $(\alpha-\alpha^2\beta)n+\frac{3\varepsilon}{10}n$ only with probability at most $\frac{\varepsilon^3}{2000}$.
\end{proof}

We are ready to prove a part of the theorem for trees with not too large maximum degree.

\begin{lemma}\label{Lemma0}
For every $\varepsilon>0$ there exists an integer $M'_{\varepsilon}$ such that if $G$ is a tree with $n\geq M'_{\varepsilon}$ vertices and the maximum degree at most $\delta_{\varepsilon}n$, then 
$$\max_{\mathcal{A}}CC(\mathcal{A})\leq\left(\frac{1}{4}+\frac{8}{10}\varepsilon\right)n.$$
\end{lemma}

\begin{proof}
Fix $\varepsilon>0$ and a tree $G$ with $n$ vertices and maximum degree at most $\delta_{\varepsilon}n$. $G$ satisfies the assumptions of Lemma \ref{LemmaConcentration} with $\beta=1-\frac{1}{n}$. Suppose $\mathcal{A}_{opt}$ is an optimal algorithm -- an algorithm that maximizes the expected value of the number of connected components at the stopping time. For every permutation $\sigma\in S_n$ algorithm $\mathcal{A}_{opt}$ stops at some time $\mathcal{A}_{opt}(\sigma)$. We consider threshold $\alpha n$ for every $\alpha\in\{0,\frac{\varepsilon}{10},2\frac{\varepsilon}{10},\dots,\lfloor\frac{10}{\varepsilon}\rfloor\frac{\varepsilon}{10}\}$. Let $S_\alpha$ be the set of all permutations $\sigma$ such that $\mathcal{A}_{opt}(\sigma)\in[\alpha n,(\alpha+\frac{\varepsilon}{10})n)$. Notice that for $\sigma\in S_\alpha$
$$CC(\sigma,\mathcal{A}_{opt}(\sigma))\leq CC(\sigma,\alpha n)+\frac{\varepsilon}{10}n,$$
and by Lemma \ref{LemmaConcentration} except for `odd' events of small probability it is bounded by
$$\leq\left(\alpha-\alpha^2\left(1-\frac{1}{n}\right)\right)n+\frac{3\varepsilon}{10}n+\frac{\varepsilon}{10}n\leq \frac{1}{4}n+1+\frac{4\varepsilon}{10}n,$$
since the maximum of the function $\alpha-\alpha^2$ is $\frac{1}{4}$ and it is attained for $\alpha=\frac{1}{2}$. The probability of all `odd' events (over all $\alpha$'s) is at most $(\lfloor\frac{10}{\varepsilon}\rfloor+1)\frac{\varepsilon^3}{2000}\leq\frac{2\varepsilon}{10}$ (the last inequality holds, because we can assume $\varepsilon<1$) and the maximum number of connected components is at most $n$. Hence, 
$$CC(\mathcal{A}_{opt})\leq \frac{1}{4}n+1+\frac{4\varepsilon}{10}n+\frac{2\varepsilon}{10}\cdot n\leq\left(\frac{1}{4}+\frac{8}{10}\varepsilon\right)n.$$
\end{proof}

Next we prove the theorem for `star-like' trees, that is for trees with at most one vertex of large degree. 

\begin{lemma}\label{Lemma1}
For every $\varepsilon>0$ there exists an integer $M''_{\varepsilon}$ such that if $G$ is a tree with $n\geq M''_{\varepsilon}$ vertices, a fixed vertex $v$ of arbitrary degree, and the maximum degree of remaining vertices at most $\delta_{\varepsilon}n$, then 
$$\max_{\mathcal{A}}CC(\mathcal{A})\leq\left(\frac{1}{4}+\frac{8}{10}\varepsilon\right)n.$$
\end{lemma}

\begin{proof}
Let $H$ be a forest constructed from $G$ by deleting edges incident to $v$. Notice that $H$ satisfies the assumptions of Lemma \ref{LemmaConcentration}. Let $\beta\in[0,1]$ be such that $|E(H)|=\beta n$.

Denote the neighborhood (or, the closed neighborhood together with a vertex) in $G$ of the vertex $v$ by $N$ (and $\overline{N}$ respectively). Clearly, $|N|=(1-\beta)n-1$. 

Suppose $\mathcal{A}_{opt}$ is an optimal algorithm for $G$ -- an algorithm that maximizes the expected number of connected components at the stopping time. Namely,
$$\max_{\mathcal{A}}CC(\mathcal{A})=CC(\mathcal{A}_{opt}).$$
Let $\mathcal{A}_{opt}^{cut}$ be the algorithm $\mathcal{A}_{opt}$ with an additional rule, that it stops at latest at a time $(1-\frac{\varepsilon}{10})n$. Namely, $\mathcal{A}_{opt}^{cut}(\sigma)=\min\{\mathcal{A}_{opt}(\sigma),(1-\frac{\varepsilon}{10})n\}$. Clearly,
$$CC(\mathcal{A}_{opt})\leq CC(\mathcal{A}_{opt}^{cut})+\frac{\varepsilon}{10}n.$$
For every permutation $\sigma\in S_n$ the algorithm $\mathcal{A}_{opt}^{cut}$ stops at some time $\mathcal{A}_{opt}^{cut}(\sigma)$. We consider thresholds (steps) $\alpha n$ for every $\alpha\in\{\frac{\varepsilon^2}{200},2\frac{\varepsilon^2}{200},\dots,\lfloor\frac{200}{\varepsilon^2}\rfloor\frac{\varepsilon^2}{200}\}$. Let $S_\alpha$ be the set of all permutations $\sigma$ such that $\mathcal{A}_{opt}^{cut}(\sigma)\in[(\alpha-\frac{\varepsilon^2}{200})n,\alpha n)$. Let $\mathcal{A}_{step}$ be a step-algorithm made from $\mathcal{A}_{opt}^{cut}$ -- an algorithm such that $\mathcal{A}_{step}(\sigma)=\alpha n$ for $\sigma\in S_{\alpha}$. Clearly, it is a stopping algorithm ($\mathcal{A}_{step}$ stops at a next step after $\mathcal{A}_{opt}^{cut}$ stops). 

Notice that for any permutation $\sigma\in S_\alpha$ we have
$$CC(\sigma,\mathcal{A}_{opt}^{cut}(\sigma))\leq \mathbb{E}_{\tau}CC(\tau,\alpha n)+\frac{\varepsilon}{10}n=\mathbb{E}_{\tau}CC(\tau,\mathcal{A}_{step}(\tau))+\frac{\varepsilon}{10}n,$$
where the expected value is taken over all permutations $\tau$ such that $\sigma(i)=\tau(i)$ for $i=1,\dots,\mathcal{A}_{opt}^{cut}(\sigma)$. Indeed, the sum of degrees of all the remaining vertices (not among $\sigma(1),\dots,\sigma(\mathcal{A}_{opt}^{cut}(\sigma))$) is at most $2n$. The number of all the remaining vertices is at least $\frac{\varepsilon}{10}n$. And, the number of vertices to come (from $\mathcal{A}_{opt}^{cut}(\sigma)$ to $\alpha n$) is at most $\frac{\varepsilon^2}{200}n$. Thus, in average, in the worst case, at most that many new connections between connected components appear (which is the decrease in the number of connected components)
$$2n\cdot\frac{\frac{\varepsilon^2}{200}n}{\frac{\varepsilon}{10}n}=\frac{\varepsilon}{10}n.$$
Therefore, just from the definition
$$CC(\mathcal{A}_{opt}^{cut})\leq CC(\mathcal{A}_{step})+\frac{\varepsilon}{10}n.$$

Now, for every $\sigma\in S_\alpha$
$$CC_G(\sigma,\mathcal{A}_{step}(\sigma))=CC_G(\sigma,\alpha n)=CC_H(\sigma,\alpha n)-E(\overline{N})(\sigma,\alpha n)+\frac{\varepsilon}{10}n,$$
where, by compatibility of our notion, $E(\overline{N})(\sigma,\alpha n)$ is the number of edges from $G[\overline{N}]$ between the first $\alpha n$ vertices in the permutation $\sigma$.

By a similar argument to the proof of Lemma \ref{LemmaConcentration} (but much simpler -- Janson's inequality can be replaced by Chebyshev's inequality) one can show that $\mathbb{E}|N[\alpha n]|=\alpha(1-\beta)n$, and that the probability that the inequality
$$|N[\alpha n]|\geq\alpha(1-\beta)n-\frac{\varepsilon}{10}n$$
is false, is at most $\frac{\varepsilon^3}{2000}$. 

Therefore, with high probability (we mean the ratio of permutations $\sigma$ for which it is true), at least $1-\frac{2\varepsilon}{10}$ (we exclude here over all considered thresholds (steps) $\alpha$ both `odd' events -- one from Lemma \ref{LemmaConcentration} and the above one), 
\[
\begin{split}
& \;CC_G(\sigma,\mathcal{A}_{step}(\sigma))=CC_H(\sigma,\alpha n)-E(\overline{N})(\sigma,\alpha n)\\
\leq & \;(\alpha-\alpha^2\beta)n+\frac{3\varepsilon}{10}n-(\alpha(1-\beta)n-\frac{\varepsilon}{10}n)\mathbb{IF}(v\in\sigma[\alpha n])\\
\leq & \;(\alpha-\alpha^2\beta)n-(\alpha(1-\beta)n)\mathbb{IF}(v\in\sigma[\alpha n])+\frac{4\varepsilon}{10}n,
\end{split}
\]
where $\mathbb{IF}(v\in\sigma[\alpha n])$ is a $0,1$-function indicating whether $v\in\{\sigma(1),\dots,\sigma(\alpha n)\}$.

Hence, taking into account $\frac{2\varepsilon}{10}$ remaining cases (and that $CC(\sigma,\mathcal{A}_{opt}(\sigma))\leq n$) we get an inequality
$$CC_G(\mathcal{A}_{step})\leq\max_{\mathcal{A}}\mathbb{M}_\beta(\mathcal{A})+\frac{6\varepsilon}{10}n,$$ 
where $\max_{\mathcal{A}}\mathbb{M}(\mathcal{A})$ is the maximum expected score of a stopping algorithm in the following `meta' game $\mathbb{M}_\beta$ for a fixed parameter $\beta$:

\smallskip

\textit{Let $V$ be a set with $n$ vertices, out of which $v$ is distinguished. Elements of $V$ become active on-line in time, one by another, in a random order (permutation) $\sigma\in S_n$. The score of the game after $\alpha n$ moves is given by the formula:}
$$(\alpha-\alpha^2\beta)n-(\alpha(1-\beta)n)\mathbb{IF}(v\in\sigma[\alpha n]).$$

\smallskip

Observe that the only information an algorithm gets during the game (that was not present at the begining of the game) is whether $v\in\sigma[\alpha n]$. Therefore, the maximum score a stopping algorithm can achieve in the meta game is realized by the following strategy $\mathcal{A}_{\alpha,\gamma}$ with parameters $\alpha,\gamma$:
\begin{itemize}
\item take exactly $\alpha n$ vertices,
\item if $v$ appears, then take exactly $\gamma n$ vertices in total.
\end{itemize}
The above strategies $\mathcal{A}_{\alpha,\gamma}$ are too general in a sense that when $v$ comes as $i$-th vertex, then it is already not possible to take $\gamma n$ vertices in total if $\gamma n<i$. However, when we are bounding the maximum score we are allowed enlarge the set of strategies. On the other hand, these strategies are also a bit too specific, since the number of vertices a strategy takes could depend on the time when $v$ comes. But, since the time $v$ comes does not impact the score, all other strategies are just convex combinations of $\mathcal{A}_{\alpha,\gamma}$ strategies. 

The score of the algorithm $\mathcal{A}_{\alpha,\gamma}$ is given by the formula
$$(1-\alpha)(\alpha-\alpha^2\beta)n+\alpha((\gamma-\gamma^2\beta)n-\gamma(1-\beta)n).$$
By Lemma \ref{Lemma4} we have
$$\max_{\mathcal{A}}\mathbb{M}_\beta(\mathcal{A})\leq\frac{1}{4}n.$$ 

Finally,
\[
\begin{split}
\max_{\mathcal{A}}CC(\mathcal{A})=CC(\mathcal{A}_{opt}) & \leq CC(\mathcal{A}_{opt}^{cut})+\frac{\varepsilon}{10}n\leq CC(\mathcal{A}_{step})+\frac{2\varepsilon}{10}n\\ 
& \leq\max_{\mathcal{A}}\mathbb{M}_\beta(\mathcal{A})+\frac{8\varepsilon}{10}n\leq\left(\frac{1}{4}+\frac{8}{10}\varepsilon\right)n.
\end{split}
\]
\end{proof}

\begin{lemma}\label{Lemma4}
Let $\alpha,\beta,\gamma$ be variables from $[0,1]$. The maximum of the function 
$$\phi(\alpha,\beta,\gamma)=(1-\alpha)(\alpha-\alpha^2\beta)+\alpha((\gamma-\gamma^2\beta)-\gamma(1-\beta))$$
is $\frac{1}{4}$, and it is attained exactly for triples $(\alpha,\beta,\gamma)=(\frac{1}{2},\beta,\frac{1}{2}),(\frac{1}{2},0,\gamma),(1,1,\frac{1}{2})$.
\end{lemma}

\begin{proof}
Simplify
$$\phi=(1-\alpha)(\alpha-\alpha^2\beta)+\alpha\beta(\gamma-\gamma^2).$$
The function $\gamma-\gamma^2$ maximizes for $\gamma=\frac{1}{2}$, then it attains value $\frac{1}{4}$. Thus, either $\alpha=0$ and $\phi=0$, or
$\beta=0$ and $\phi=(1-\alpha)\alpha\leq\frac{1}{4}$ attaining the maximum value for $(\frac{1}{2},0,\gamma)$, or $\gamma=\frac{1}{2}$ and
$$\phi=(1-\alpha)(\alpha-\alpha^2\beta)+\frac{1}{4}\alpha\beta=\alpha-\alpha^2\beta-\alpha^2+\alpha^3\beta+\frac{1}{4}\alpha\beta.$$
Consider $\phi$ as a polynomial of degree $3$ in $\alpha$ with positive leading coefficient $\beta$, and the derivative
$$\frac{d\phi}{d\alpha}=1-2\alpha\beta-2\alpha+3\alpha^2\beta+\frac{1}{4}\beta=\left(\alpha-\frac{1}{2}\right)\left(3\alpha\beta-\frac{1}{2}\beta-2\right).$$
Notice that $\left(3\alpha\beta-\frac{1}{2}\beta-2\right)$ for $\alpha=\frac{1}{2}$ has value $2\beta-2\leq 0$, thus the linear form $\left(3\alpha\beta-\frac{1}{2}\beta-2\right)$ has a root $\alpha\geq\frac{1}{2}$.

A polynomial of degree $3$ with positive leading coefficient attains its maximum on an interval either in the smaller root of its derivative, or at the endpoints of the interval. We have

$$\phi\left(0,\beta,\frac{1}{2}\right)=0,\:\;\;\;\phi\left(1,\beta,\frac{1}{2}\right)=\frac{1}{4}\beta,\;\:\;\;\phi\left(\frac{1}{2},\beta,\frac{1}{2}\right)=\frac{1}{4},$$
thus $\phi\leq\frac{1}{4}$ attaining the maximum value for $(\alpha,\beta,\gamma)=(\frac{1}{2},\beta,\frac{1}{2}),(1,1,\frac{1}{2})$.
 \end{proof}

\begin{proofofmain}	
Fix $\varepsilon>0$ and a tree $G$ with $n$ vertices.

Let $v_1,\dots,v_k$ be all vertices in $G$ of degree at least $\delta_{\varepsilon}\frac{\varepsilon}{10}n$. Notice that $k<\frac{20}{\varepsilon\delta_{\varepsilon}}$. For every $v_i$ consider $G$ as a rooted tree with $v_i$ being the root. For every edge $e$ incident to $v_i$, if the subtree of $G$ attached to the root by $e$ has more than $\delta_{\varepsilon}^2\frac{\varepsilon}{10}n$ vertices, add $e$ to the set $R$.
The set $R$ has at most $k\cdot\frac{10}{\varepsilon\delta_{\varepsilon}^2}<\frac{20}{\varepsilon\delta_{\varepsilon}}\cdot\frac{10}{\varepsilon\delta_{\varepsilon}^2}$ edges in total (over all $v_i$ and $e$).

Consider connected components $H_1,\dots,H_m$ of the graph obtained from $G$ by removing edges from $R$. No two vertices $v_1,\dots,v_k$ belong to the same component, so without loss of generality assume that $v_i\in H_i$. Let $H$ be a tree obtained from trees $H_{k+1},\dots,H_m$ by joining them one by another with edges attached to their leaves (in order not to increase their maximum degree). 

Consider $H$. Its maximum degree is at most $\delta_{\varepsilon}\frac{\varepsilon}{10}n$, as $H$ does not contain any of vertices $v_1,\dots,v_k$. If $\vert H\vert\geq\frac{\varepsilon}{10}n$, then $H$ satisfies the assumptions of Lemma \ref{Lemma0}, and therefore 
$$\max_{\mathcal{A}}CC_{H}(\mathcal{A})\leq\left(\frac{1}{4}+\frac{8}{10}\varepsilon\right)\vert H\vert.$$ 
Otherwise, when $\vert H\vert<\frac{\varepsilon}{10}n$, we have a bound $\max_{\mathcal{A}}CC_{H}(\mathcal{A})<\frac{\varepsilon}{10}n$. Hence, together
$$\max_{\mathcal{A}}CC_{H}(\mathcal{A})\leq\left(\frac{1}{4}+\frac{8}{10}\varepsilon\right)\vert H\vert+\frac{\varepsilon}{10}n.$$ 

Consider $H_i$ (for $i=1,\dots,k$). It contains a vertex $v_i$ of degree at least $\delta_{\varepsilon}\frac{\varepsilon}{10}n$ and degrees of the remaining vertices are at most $\delta_{\varepsilon}^2\frac{\varepsilon}{10}n$. Hence $H_i$ satisfies the assumptions of Lemma \ref{Lemma1}, and therefore 
$$\max_{\mathcal{A}}CC_{H_i}(\mathcal{A})\leq\left(\frac{1}{4}+\frac{8}{10}\varepsilon\right)\vert H_i\vert.$$ 

We have a bound
\[
\begin{split}
\max_{\mathcal{A}}CC_G(\mathcal{A}) & \leq\sum_{i=1}^k\max_{\mathcal{A}}CC_{H_i}(\mathcal{A})+\max_{\mathcal{A}}CC_{H_{k+1}\cup\dots\cup H_m}(\mathcal{A})\\
& \leq\sum_{i=1}^k\max_{\mathcal{A}}CC_{H_i}(\mathcal{A})+\max_{\mathcal{A}}CC_{H}(\mathcal{A})+\vert R\vert\\
& \leq\sum_{i=1}^k\left(\frac{1}{4}+\frac{8}{10}\varepsilon\right)\vert H_i\vert+\left(\frac{1}{4}+\frac{8}{10}\varepsilon\right)\vert H\vert+\frac{\varepsilon}{10}n+\vert R\vert\\
& =\left(\frac{1}{4}+\frac{8}{10}\varepsilon\right)n+\frac{\varepsilon}{10}n+\vert R\vert \leq\left(\frac{1}{4}+\varepsilon\right)n
\end{split}
\]
for sufficiently large $n$. 
\kwadrat
\end{proofofmain}

\section{$k$-Trees}\label{kTree}

Recall that a graph $G$ is \emph{$k$-degenerate} if every induced subgraph of $G$ has a vertex of degree at most $k$. Equivalently, the \emph{coloring number} of $G$ is at most $k+1$ -- that is, there exists an ordering of vertices of $G$ such that each vertex is joined to at most $k$ vertices that are earlier in the ordering. 

A graph $G$ is \emph{maximal $k$-degenerate} if, as its name says, it is maximal with respect to the inclusion of edges among $k$-degenerate graphs on a given vertex set. Equivalently, there exists an ordering  $v_1,\dots,v_n$ of vertices of $G$ such that the first $k$ vertices form a clique $K_{k}$, and each of the remaining vertices is joined to exactly $k$ vertices that are earlier in the ordering. This ordering is not unique, however we fix one and consider maximal $k$-degenerate graphs as equipped with such an ordering. For a vertex $v$ (not from the initial $K_k$) denote by $M_v$ the set of exactly $k$ neighbors that are earlier in the ordering, for vertices of the initial $K_k$ we set $M_{v_i}=\{v_1,\dots,v_{i-1}\}$.

A graph $G$ is \emph{$k$-tree} if it is maximal $k$-degenerate, and sets $M_v$ are cliques. The notion of a $k$-tree is inseparably connected with another well-known graph parameter -- treewidth. Namely, $k$-trees are exactly maximal (w.r.t. the inclusion of edges) graphs with treewidth equal to $k$.  

\subsection*{Blind variant}\label{kBlind}

\begin{theorem}\label{kTheoremBlind}
Suppose $G$ is a $k$-tree with $n$ vertices. An almost (up to a constant) optimal strategy is to stop after $l=\frac{1}{k+1}n$ vertices. We have
$$\frac{k^k}{(k+1)^{k+1}}n-\frac{(k+2)}{e}<CC\left[\frac{1}{k+1}n\right]\leq\max_{l}CC[l]<\frac{k^k}{(k+1)^{k+1}}n+1.$$
\end{theorem}

\begin{proof}
The stopping function $l$ can depend only on $n$, since the algorithm does not have any other information. Suppose the algorithm stops after $l(n)=\alpha(n)n$ steps and $F$ is the set of active vertices. 

Suppose $C$ is an induced subgraph of a $k$-tree $G$, and $C$ is connected. Observe that, since $G$ is a $k$-tree, $C$ has exactly one vertex $v$ such that $M_v\cap C=\emptyset$. Namely, $v$ has to be the least vertex of $C$ in the $k$-tree ordering. Hence, the number of connected components in an induced subgraph $F$ of $G$ equals to the number of vertices $v$ in $F$ such that $M_v\cap F=\emptyset$. For $v\in V(G)$ let $X_v$ be a random variable $X_v=\mathbb{IF}(v\in F\text{ and }M_v\cap F=\emptyset)$. We have 
\[
\begin{split}
CC(F)= & \sum_{v\in V(G)}X_v,\\
\mathbb{E}(CC[l])= & \sum_{v\in V(G)}\mathbb{E}(X_v).
\end{split}
\]
For a vertex $v$ (not from the initial $k$-clique) we have 
$$\mathbb{E}(X_v)=\frac{n-l}{n}\cdot\frac{n-l-1}{n-1}\dots\frac{n-l-k+1}{n-k+1}\cdot\frac{l}{n-k}\leq\frac{n-l}{n}\cdot\frac{n-l}{n}\dots\frac{n-l}{n}\cdot\frac{l}{n-k}=$$
$$=\left(\frac{n-\alpha n}{n}\right)^k\cdot\frac{\alpha n}{n-k}=(1-\alpha)^k\alpha\frac{n}{n-k}.$$
Altogether, taking into account that on the first $k$ vertices there is at most one connected component, we have a bound
$$\mathbb{E}(CC[l])=\sum_{v\in V(G)}\mathbb{E}(X_v)<1+(n-k)(1-\alpha)^k\alpha\frac{n}{n-k}=(1-\alpha)^k\alpha n+1.$$
The function $(1-\alpha)^k\alpha$ maximizes for $\alpha=\frac{1}{k+1}$, so
$$\mathbb{E}(CC[l])<\frac{k^k}{(k+1)^{k+1}}n+1.$$
	
Now, consider $l=\frac{1}{k+1}n$ ($\alpha=\frac{1}{k+1}$). We have
\[
\begin{split}
\mathbb{E}(X_v)= & \frac{n-l}{n}\cdot\frac{n-l-1}{n-1}\dots\frac{n-l-k+1}{n-k+1}\cdot\frac{l}{n-k}\\
\geq & \frac{n-l-k}{n-k}\cdot\frac{n-l-k}{n-k}\dots\frac{n-l-k}{n-k}\cdot\frac{l}{n-k}\\
\geq & \left(\frac{n-\alpha n}{n}\right)^k\cdot\frac{\alpha n}{n}\cdot\left(\frac{n-l-k}{n-l}\right)^k\\
= & (1-\alpha)^k\cdot\alpha\cdot\left(1-\frac{k+1}{n}\right)^k\\
\geq & (1-\alpha)^k\cdot\alpha\cdot\left(1-\frac{k(k+1)}{n}\right).
\end{split}
\]
Hence,
$$CC\left[\frac{1}{k+1}n\right]\geq (n-k)(1-\alpha)^k\cdot\alpha\cdot\left(1-\frac{k(k+1)}{n}\right)\geq$$
$$\frac{k^k}{(k+1)^{k+1}}n-\frac{(k+2)k^{k+1}}{(k+1)^{k+1}}\geq\frac{k^k}{(k+1)^{k+1}}n-\frac{(k+2)}{e}.$$
\end{proof}

\subsection*{Full information variant}

We prove that for $k$-trees there is no asymptotically better algorithm than `wait until $\frac{1}{k+1}$ fraction of vertices'. Thus, the best result for $k$-trees is $\left(\frac{k^k}{(k+1)^{k+1}}+o(1)\right)n$ connected components. 

We begin with a concentration lemma on a slightly more general structure. Let $V$ be an $n$-element vertex set. By a \emph{$k$-system} $G$ we mean a set of pairs $(v,M_v)$ such that $v\in V$, $M_v$ is a $k$-element subset of $V$, $v\notin M_v$, every $v\in V$ appears in at most one pair $(v,M_v)$ in $G$. The \emph{degree} of a vertex $v\in V$ in $G$ is the number of sets $M_w$ to which $v$ belongs. The \emph{maximum degree} of a $k$-system is the maximum degree of a vertex. Notice that if $G$ is a maximal $k$-degenerate graph, then the set of pairs 
$$\{(v,M_v):v\in V(G)\setminus K_{k}\}$$
is a $k$-system. These systems are supposed to generalize maximal $k$-degenerate graphs.

Suppose elements of the ground set $V$ of a $k$-system $G$ become active on-line in time, one by another, in a random order (permutation) $\sigma\in S_n$. A pair $(v,M_v)$ of $G$ is \emph{active} exactly when $v$ is active and vertices of $M_v$ are not active. Denote by $WV_G(\sigma,t)$ the number of active pairs of $G$ at a time $t$ on a permutation $\sigma$.

\begin{lemma}\label{kLemmaMain}
Let $k$ be a nonnegative integer. For every $\varepsilon>0$ there exist numbers $N_{k,\varepsilon},c_k,d_k$ such that if $n\geq N_{k,\varepsilon}$, then for every $k$-system $G$ there exists a set of vertices $D\subset V$ and a partition 
$$G=\bigsqcup_{T\subset D,\vert T\vert\leq k}G_T$$ 
such that for $(v,M_v)\in G_T$, $T\subset M_v$, and for every $\alpha\in[0,1]$
$$WV_G(\sigma,\alpha n)<\sum_{T\subset D,\vert T\vert\leq k}\mathbb{IF}(T\cap\sigma[\alpha n]=\emptyset)(1-\alpha)^{k-\vert T\vert}\alpha\vert G_T\vert+c_k\varepsilon n$$
with probability greater than $1-d_k\varepsilon$ (the ratio of $\sigma$'s for which the inequality holds).
\end{lemma}

\begin{proof}
We prove the lemma by induction on $k$. 

For $k=0$ the assertion with $c_0=d_0=1$ reduces to the following statement -- if an $\alpha n$-element subset $S$ of $V$ is taken at random, then 
$$\mathbb{P}(\vert G\cap S\vert<\alpha\vert G\vert+\varepsilon n)>1-\varepsilon,$$ when $n$ is large enough. This follows easily from Chebyshev's inequality.

Suppose now that the lemma holds for all nonnegative integers less than $k$. Fix $k,\varepsilon$, and let $H$ be an arbitrary $k$-system on the ground set $V$ with the maximum degree at most $\delta|H|$. 
We want to get a bound on $WV_H(\sigma,\alpha n)$. Suppose that $\alpha\geq 2\varepsilon$, and that $\alpha n$-element subset $S$ of $V$ is chosen at random, i.e. such that every set of $\alpha n$ vertices is equally probable. Such a random choice of $S$ can be done by a following procedure:
\begin{enumerate} 
\item drawn subset $R$ of $V$ randomly, such that each vertex is taken independently with probability $\alpha-\varepsilon$,
\item already from the binomial distribution it follows that 
$$\mathbb{P}\left(|R|\in\left((\alpha-2\varepsilon)n,\alpha n\right)\right)\rightarrow_{n\rightarrow\infty}1,$$
so for sufficiently large $n$ this happens with probability at least $\frac{1}{2}$,
\item if the size of $R$ is not from the interval $\left((\alpha-2\varepsilon)n,\alpha n\right)$, resample $R$,
\item if the size of $R$ is from the interval $\left((\alpha-2\varepsilon)n,\alpha n\right)$, add $\alpha n-|R|<2\varepsilon n$ vertices at random (such that every set is equally probable) and return $S$.  
\end{enumerate}
Since inclusions of individual elements in $R$ are independent, then so are in $V\setminus R$. Thus, we may use Janson's inequality (Theorem \ref{TheoremJanson}) for a random set $V\setminus R$ twice -- once with  $\mathcal{A}_1=\{M_v:(v,M_v)\in H\}$ and secondly $\mathcal{A}_2=\{\{v\}\cup M_v:(v,M_v)\in H\}$.

In the first case, denote $X_1=\{v:M_v\subset V\setminus R\}$, and $x_1=|X_1|$. We want to get a bound on $\mathbb{P}((x_1>(1+\varepsilon^2)\mathbb{E}x_1)$. The random variable $x_1$ is nonnegative, so
$$\mathbb{P}\left(x_1>(1+\varepsilon^2)\mathbb{E}x_1\right)\cdot\varepsilon^2\mathbb{E}x_1\leq 1\cdot\varepsilon^3\mathbb{E}x_1+\mathbb{P}\left(x_1\leq\left(1-\varepsilon^3\right)\mathbb{E}x_1\right)\cdot \mathbb{E}x_1.$$ 
Hence,
\[
\begin{split}
\mathbb{P}\left(x_1>(1+\varepsilon^2)\mathbb{E}x_1\right)\leq & \;\frac{\varepsilon^3+\mathbb{P}\left(x_1\leq\left(1-\varepsilon^3\right)\mathbb{E}x_1\right)}{\varepsilon^2}\\
\leq & \;\varepsilon+\frac{1}{\varepsilon^2}exp\left(-\frac{1}{2}\frac{\varepsilon^6(\mathbb{E}x_1)^2}{\mathbb{E}x_1+2\delta(\mathbb{E}x_1)^2}\right).
\end{split}
\]
We have $\mathbb{E}x_1=(1-\alpha+\varepsilon)^k\vert H\vert$, so for sufficiently large $|H|$ and small $\delta$ the probability $\mathbb{P}((x_1>(1+\varepsilon^2)\mathbb{E}x_1)$ is smaller than $2\varepsilon$.

In the second case, denote $X_2=\{v:\{v\}\cup M_v\subset V\setminus R\}$, and $x_2=|X_2|$. We want to get a bound on $\mathbb{P}(x_2<(1-\varepsilon^2)\mathbb{E}x_2)$. $\mathbb{E}x_2=(1-\alpha+\varepsilon)^{k+1}\vert H\vert$, and
$$\mathbb{P}\left(x_2\leq\left(1-\varepsilon^2\right)\mathbb{E}x_2\right)\leq exp\left(-\frac{1}{2}\frac{\varepsilon^4(\mathbb{E}x_2)^2}{\mathbb{E}x_2+2\delta(\mathbb{E}x_2)^2}\right)$$
so for sufficiently large $|H|$ and small $\delta$ this probability is smaller than $2\varepsilon$.

Denote by $\delta_{k,\varepsilon}$ the value such that for every $\delta<\delta_{k,\varepsilon}$ and sufficiently large $|H|$ both probabilities are smaller than $2\varepsilon$. Then, we have
\[
\begin{split}
2\varepsilon+2\varepsilon & >\mathbb{P}((x_1>(1+\varepsilon^2)\mathbb{E}x_1)+\mathbb{P}(x_2<(1-\varepsilon^2)\mathbb{E}x_2)\\
& \geq\mathbb{P}((x_1>(1+\varepsilon(\alpha-\varepsilon))\mathbb{E}x_1)+\mathbb{P}(x_2<\mathbb{E}x_2-\varepsilon(\alpha-\varepsilon)\mathbb{E}x_2)\\
& \geq\mathbb{P}((x_1>(1+\varepsilon(\alpha-\varepsilon))\mathbb{E}x_1)+\mathbb{P}(x_2<\mathbb{E}x_2-\varepsilon(\alpha-\varepsilon)\mathbb{E}x_1)\\
& =\mathbb{P}((x_1-\mathbb{E}x_1)>\varepsilon(\alpha-\varepsilon)\mathbb{E}x_1)+\mathbb{P}(-(x_2-\mathbb{E}x_2)>\varepsilon(\alpha-\varepsilon)\mathbb{E}x_1)\\
& \geq\mathbb{P}((x_1-\mathbb{E}x_1)-(x_2-\mathbb{E}x_2)>2\varepsilon(\alpha-\varepsilon)\mathbb{E}x_1)\\
& =\mathbb{P}((x_1-\mathbb{E}x_1)-(x_2-\mathbb{E}x_2)>2\varepsilon(\mathbb{E}x_1-\mathbb{E}x_2))\\
& =\mathbb{P}(x_1-x_2>(1+2\varepsilon)(\mathbb{E}x_1-\mathbb{E}x_2))\\
& =\mathbb{P}(x_1-x_2>(1+2\varepsilon)((1-\alpha+\varepsilon)^{k}-(1-\alpha+\varepsilon)^{k+1})\vert H\vert)\\
& =\mathbb{P}(x_1-x_2>(1+2\varepsilon)(\alpha-\varepsilon)(1-\alpha+\varepsilon)^{k}\vert H\vert)\\
& \geq\mathbb{P}(x_1-x_2>\alpha(1-\alpha)^{k}\vert H\vert+2^k\varepsilon\vert H\vert).
\end{split}
\]

Notice that $$WV_{H}(R)=x_1-x_2,$$
therefore after adding at most $2\varepsilon n$ vertices we get a bound
$$WV_{H}(\sigma,\alpha n)<\alpha(1-\alpha)^k\vert H\vert+2^k\varepsilon\vert H\vert+2\varepsilon n\leq \alpha(1-\alpha)^k\vert H\vert+(2^k+2)\varepsilon n$$
with probability greater than $1-8\varepsilon$ (the probability of failure is doubled because of point $(2)$). Now, when $\alpha<2\varepsilon$ we have that 
$$\mathbb{P}(\vert S\vert>2\alpha n)<\varepsilon$$
when $n$ is large enough. Thus,
$$WV_{H}(\sigma,\alpha n)\leq 2\alpha n<4\varepsilon n$$
with probability greater than $1-\varepsilon$.

Let $G$ be a $k$-system. Let $v_1,\dots,v_l$ be all vertices in $G$ of degree at least $\varepsilon\delta_{k,\varepsilon}n$. Notice that $l<\frac{k}{\varepsilon\delta_{k,\varepsilon}}$. Define $k$-systems $G_i$ inductively
$$G_i=\{(v,M_v)\in G\setminus(G_1\cup\dots\cup G_{i-1}):v_i\in M_v\},$$
and $H=G\setminus(G_1\cup\dots\cup G_l)$. Clearly, $G=G_1\sqcup\dots\sqcup G_l\sqcup H$.

Notice that either $|H|\geq\varepsilon n$, so $H$ has maximum degree at most $\delta_{k,\varepsilon}|H|$, and the above bounds are satisfied for $H$. Or, otherwise $|H|<\varepsilon n$ and $WV_{H}(\sigma,\alpha n)<\varepsilon n$.

Observe that for every $G_i$ the following
$$\overline{G_i}=\{(v,M_v\setminus\{v_i\}):(v,M_v)\in G_i\}$$
is a $(k-1)$-system. Moreover,
$$WV_{G_i}(\sigma,\alpha n)=\mathbb{IF}(v_i\notin\sigma[\alpha n])WV_{\overline{G_i}}(\sigma,\alpha n).$$
Hence, from the inductive assumption we get $D_i\subset V$ and a partition 
$$\overline{G_i}=\bigsqcup_{T_i\subset D_i,\vert T_i\vert\leq k-1}\overline{G_i}_{T_i}$$ such that
$$WV_{\overline{G_i}}(\sigma,\alpha n)<\sum_{T_i\subset D_i,\vert T_i\vert\leq k-1}\mathbb{IF}(T_i\cap\sigma[\alpha n]=\emptyset)(1-\alpha)^{k-1-\vert T_i\vert}\alpha\vert \overline{G_i}_{T_i}\vert+c_{k-1}\varepsilon n$$
with probability greater than $1-d_{k-1}\varepsilon$. 

For $G$ define $D=\bigcup_{i=1}^lD_i\cup\{v_i\}$, and define $G_{\emptyset}=H$, and for $\emptyset\neq T\subset D$: $G_T\neq\emptyset$ only when $v_i\in T$ and $T_i=T\setminus\{v_i\}\subset D_i$ for some $i$, then  
$$G_T=\{(v,M_v)\in G: (v,M_v\setminus\{v_i\})\in\overline{G_i}_{T_i}\}.$$

We are ready to show the inequality from the assertion
$$WV_G(\sigma,\alpha n)=\sum_{i=1}^lWV_{G_i}(\sigma,\alpha n)+WV_{H}(\sigma,\alpha n)$$
$$=\sum_{i=1}^l\mathbb{IF}(v_i\notin\sigma[\alpha n])WV_{\overline{G_i}}(\sigma,\alpha n)+WV_{H}(\sigma,\alpha n)$$
$$<\sum_{i=1}^l\sum_{T_i\subset D_i,\vert T_i\vert\leq k-1}\mathbb{IF}(v_i\notin\sigma[\alpha n])\mathbb{IF}(T_i\cap\sigma[\alpha n]=\emptyset)(1-\alpha)^{k-1-\vert T_i\vert}\alpha\vert \overline{G_i}_{T_i}\vert+c_{k-1}\varepsilon n$$
$$+\alpha(1-\alpha)^k\vert H\vert+(2^k+2)\varepsilon n=$$
\smallskip
$$=\sum_{T\subset D,\vert T\vert\leq k}\mathbb{IF}(T\cap\sigma[\alpha n]=\emptyset)(1-\alpha)^{k-\vert T\vert}\alpha\vert G_T\vert+\left(\frac{k}{\varepsilon\delta_{k,\varepsilon}}c_{k-1}+2^k+2\right)\varepsilon n$$
with probability greater than $1-\left(\frac{k}{\varepsilon\delta_{k,\varepsilon}}d_{k-1}+8\right)\varepsilon$. We get the inductive assertion with $c_k=\frac{k}{\varepsilon\delta_{k,\varepsilon}}c_{k-1}+2^k+2$ and $d_k=\frac{k}{\varepsilon\delta_{k,\varepsilon}}d_{k-1}+8$.
\end{proof}

\begin{theorem}\label{kTheoremMain}
For every $\varepsilon>0$ there exists an integer $N_{\varepsilon}$ such that if $G$ is a maximal $k$-degenerate graph with $n\geq N_{\varepsilon}$ vertices, then 
$$\max_{\mathcal{A}}CC(\mathcal{A})\leq\left(\frac{k^k}{(k+1)^{k+1}}+\varepsilon\right)n.$$
\end{theorem}

\begin{proof}
We begin with an easy observation which changes the object we study. Suppose $C$ is a connected induced subgraph of a maximal $k$-degenerate graph $G$. Observe that $C$ has at least one vertex $v$ such that $M_v\cap C=\emptyset$. We call such a vertex a \emph{witnessing vertex} of the component. Indeed, the least vertex of $C$ in the maximal $k$-degenerate graph ordering has this property. Hence, the number of connected components in an induced subgraph $F$ of $G$ is less or equal to the number of witnessing vertices $v$ in $F$, that is $v\in F$ such that $M_v\cap F=\emptyset$. We denote by $WV(F)$ the set of all witnessing vertices in $F$. Hence,
$$\max_{\mathcal{A}}CC(\mathcal{A})\leq\max_{\mathcal{A}}WV(\mathcal{A}).$$

Suppose $\mathcal{A}_{opt}$ is an optimal algorithm for $WV$ in $G$ -- an algorithm that maximizes the expected number of witnessing vertices at the stopping time. Namely,
$$\max_{\mathcal{A}}WV(\mathcal{A})=WV(\mathcal{A}_{opt}).$$
Let $\mathcal{A}_{opt}^{cut}$ be the algorithm $\mathcal{A}_{opt}$ with an additional rule, that it stops at latest at a time $(1-\frac{\varepsilon}{10})n$. Namely, $\mathcal{A}_{opt}^{cut}(\sigma)=\min\{\mathcal{A}_{opt}(\sigma),(1-\frac{\varepsilon}{10})n\}$. Clearly,
$$WV(\mathcal{A}_{opt})\leq WV(\mathcal{A}_{opt}^{cut})+\frac{\varepsilon}{10}n.$$
For every permutation $\sigma\in S_n$ the algorithm $\mathcal{A}_{opt}^{cut}$ stops at some time $\mathcal{A}_{opt}^{cut}(\sigma)$. We consider thresholds (steps) $\alpha n$ for every $\alpha\in\{\frac{\varepsilon^2}{100k},2\frac{\varepsilon^2}{100k},\dots,\lfloor\frac{100k}{\varepsilon^2}\rfloor\frac{\varepsilon^2}{100k}\}$. Let $S_\alpha$ be the set of all permutations $\sigma$ such that $\mathcal{A}_{opt}^{cut}(\sigma)\in[(\alpha-\frac{\varepsilon^2}{100k})n,\alpha n)$. Let $\mathcal{A}_{step}$ be a step-algorithm from $\mathcal{A}_{opt}^{cut}$ -- an algorithm such that $\mathcal{A}_{step}(\sigma)=\alpha n$ for $\sigma\in S_{\alpha}$. Clearly, it is a stopping algorithm ($\mathcal{A}_{step}$ stops at a next step after a stopping algorithm $\mathcal{A}_{opt}^{cut}$ stops). 

Notice that for any permutation $\sigma\in S_\alpha$ we have
$$WV(\sigma,\mathcal{A}_{opt}^{cut}(\sigma))\leq \mathbb{E}_{\tau}WV(\tau,\alpha n)+\frac{2\varepsilon}{10}n=\mathbb{E}_{\tau}WV(\tau,\mathcal{A}_{step}(\tau))+\frac{2\varepsilon}{10}n,$$
where the expected value is taken over all permutations $\tau$ such that $\sigma(i)=\tau(i)$ for $i=1,\dots,\mathcal{A}_{opt}^{cut}(\sigma)$. Indeed, the sum of degrees of all the remaining vertices (not among $\sigma(1),\dots,\sigma(\mathcal{A}_{opt}^{cut}(\sigma))$) is at most $2kn$. The number of all the remaining vertices is at least $\frac{\varepsilon}{10}n$. And, the number of vertices to come (from $\mathcal{A}_{opt}^{cut}(\sigma)$ to $\alpha n$) is at most $\frac{\varepsilon^2}{100k}n$. Thus, in average, in the worst case, at most that many witnessing vertices $v$ disappear (because a vertex in their set $M_v$ appears)
$$2kn\cdot\frac{\frac{\varepsilon^2}{100k}n}{\frac{\varepsilon}{10}n}=\frac{2\varepsilon}{10}n.$$

Therefore, just from the definition

$$WV(\mathcal{A}_{opt}^{cut})\leq WV(\mathcal{A}_{step})+\frac{2\varepsilon}{10}n.$$

Consider $G$ as an obvious $k$-system, namely as a set of pairs $(v,M_v)$. Clearly, the number $WV_G(\sigma,\alpha n)$ means the same when $G$ is a maximal $k$-degenerate graph, and when it is a $k$-system. By Lemma \ref{kLemmaMain} applied for $\epsilon=\min\{\frac{\varepsilon^3}{1000kd_k},\frac{\varepsilon}{10c_k}\}$ there exists a set of vertices $D$ in $G$ and a partition 
$$G=\bigsqcup_{T\subset D,\vert T\vert\leq k}G_T$$ 
such that for $(v,M_v)\in G_T$, inclusion $T\subset M_v$ holds, and for every $\alpha\in[0,1]$
$$WV_G(\sigma,\alpha n)<\sum_{T\subset D,\vert T\vert\leq k}\mathbb{IF}(T\cap\sigma[\alpha n]=\emptyset)(1-\alpha)^{k-\vert T\vert}\alpha\vert G_T\vert+c_k\epsilon n$$
with probability of an `odd' event of failure at most $d_k\epsilon$. Therefore, with high probability (we mean the ratio of permutations $\sigma$ for which it is true), at least $1-\frac{\varepsilon}{10}$ (we exclude here over all considered thresholds (steps) $\alpha$ `odd' events), 
$$WV_G(\sigma,\alpha n)<\sum_{T\subset D,\vert T\vert\leq k}\mathbb{IF}(T\cap\sigma[\alpha n]=\emptyset)(1-\alpha)^{k-\vert T\vert}\alpha\vert G_T\vert+\frac{\varepsilon}{10} n.$$

The algorithm $\mathcal{A}_{step}$ is a stopping algorithm with values in our thresholds (steps). Hence, the above inequality holds for its stopping times. Taking into account $\frac{\varepsilon}{10}$ remaining cases we get an inequality
$$WV(\mathcal{A}_{step})\leq\max_{\mathcal{A}}\mathbb{M}(\mathcal{A})+\frac{2\varepsilon}{10}n,$$ 
where $\max_{\mathcal{A}}\mathbb{M}(\mathcal{A})$ is the maximum expected score of a stopping algorithm in the following `meta' game $\mathbb{M}$:

\smallskip

\textit{Elements of an $n$-element set $V$ become active on-line in time, one by another, in a random order (permutation) $\sigma\in S_n$. The score after $\alpha n$ moves is given by the formula:}
$$\sum_{T\subset D,\vert T\vert\leq k}\mathbb{IF}(T\cap\sigma[\alpha n]=\emptyset)(1-\alpha)^{k-\vert T\vert}\alpha\vert G_T\vert.$$

\smallskip

Obviously, the score in the meta game $\mathbb{M}$ is bounded from above by a sum
$$\max_{\mathcal{A}}\mathbb{M}(\mathcal{A})\leq\sum_{T\subset D,\vert T\vert\leq k}\max_{\mathcal{A}}\mathbb{M}_T(\mathcal{A}),$$ 
where $\mathbb{M}_T$ is an analogous meta game with the score after $\alpha n$ moves given by the formula:
$$\mathbb{IF}(T\cap\sigma[\alpha n]=\emptyset)(1-\alpha)^{k-\vert T\vert}\alpha\vert G_T\vert.$$

Observe that the only information a stopping algorithm gets during the $\mathbb{M}_T$ game (that was not present at the beginning of the game, and that impacts the score) is whether $T\cap\sigma[\alpha n]=\emptyset$. However, when $T\cap\sigma[\alpha n]\neq\emptyset$, then the present and future scores are $0$. Therefore, the maximum score a stopping algorithm can achieve in the meta game $\mathbb{M}_T$ is realized by the following strategy $\mathcal{A}_{\alpha}$ with parameter $\alpha$:
\begin{itemize}
\item take exactly $\alpha n$ vertices.
\end{itemize}
The score of the algorithm $\mathcal{A}_{\alpha}$ is given by the formula
$$(1-\alpha)^{\vert T\vert}(1-\alpha)^{k-\vert T\vert}\alpha\vert G_T\vert=(1-\alpha)^{k}\alpha\vert G_T\vert,$$
which maximizes for $\alpha=\frac{1}{k+1}$. So, the maximum expected score in the meta game $\mathbb{M}_T$ satisfies
$$\max_{\mathcal{A}}\mathbb{M}_T(\mathcal{A})\leq \left(\frac{k^k}{(k+1)^{k+1}}\right)\vert G_T\vert.$$
As a consequence, the score in the meta game $\mathbb{M}$ is at most 
$$\sum_{T\subset D,\vert T\vert\leq k}\left(\frac{k^k}{(k+1)^{k+1}}\right)\vert G_T\vert=\left(\frac{k^k}{(k+1)^{k+1}}\right)\vert G\vert=\left(\frac{k^k}{(k+1)^{k+1}}\right)(n-k).$$

Finally,
\[
\begin{split}
\max_{\mathcal{A}}CC(\mathcal{A})\leq & \;\max_{\mathcal{A}}WV(\mathcal{A})=WV(\mathcal{A}_{opt})\leq WV(\mathcal{A}_{opt}^{cut})+\frac{\varepsilon}{10}n\\
\leq & \;WV(\mathcal{A}_{step})+\frac{2\varepsilon}{10}n+\frac{\varepsilon}{10}n=WV(\mathcal{A}_{step})+\frac{3\varepsilon}{10}n\\
\leq & \;\max_{\mathcal{A}}\mathbb{M}(\mathcal{A})+\frac{2\varepsilon}{10}n+\frac{3\varepsilon}{10}n\leq \left(\frac{k^k}{(k+1)^{k+1}}+\frac{5}{10}\varepsilon\right)n.
\end{split}
\]
\end{proof}

We proved that for $k$-trees there is no asymptotically better algorithm than wait until some number of vertices, that is asymptotically full information does not give any advantage compared to blindness. One could wonder if this is the case for all graphs. The answer is no.

\begin{remark}\label{RemarkFI>B}
There exist families of graphs $\{G_n\}_{n\in\N}$ for which an optimal full information algorithm is asymptotically better than an optimal blind algorithm. That is,
$$\liminf_{n\rightarrow\infty}\frac{\max_{\mathcal{A}-FI\;alg}CC_{G_n}(\mathcal{A})}{\vert G_n\vert}>\limsup_{n\rightarrow\infty}\frac{\max_{\mathcal{A}-B\;alg}CC_{G_n}(\mathcal{A})}{\vert G_n\vert}.$$	
\end{remark}

A $k$-tree $G$ is called a \emph{$k$-star} if all vertices are joined to the initial clique $K_k$, that is if for all vertices $v$ (not from the initial $K_k$) $M_v=K_k$. Clearly, in a $k$-star every vertex from the initial clique $K_k$ is joined to every other vertex.

Consider a graph $G$ constructed by joining by an edge a $2$-star $S_2$ containing $\frac{999}{1000}n$ vertices with a star $S_1$ containing $\frac{1}{1000}n$ vertices. An almost optimal strategy:
\begin{itemize} 
\item take exactly $\frac{1}{3} n$ vertices, 
\item if at least one of the two vertices of the initial clique $K_2$ of $S_2$ arrives, then take exactly $\frac{1}{2}n$ vertices in total.
\end{itemize}
In the first case we want to maximize the number of connected components in $S_2$. But, when a vertex of the initial clique $K_2$ comes, then there is only one connected component in $S_2$. Then, we start maximizing the number of connected components in $S_1$. There is no blind (fixed number of vertices stopping time) strategy asymptotically as good as this one.

\section{Open problems}

We know already an asymptotically best algorithm, however we still do not know an exact optimal stopping algorithm in the full information variant.

\begin{question}
How does an optimal stopping algorithm on a fixed tree behave? 
\end{question}

It is natural to ask if we can get a better bound than in Theorem \ref{TheoremMain}.

\begin{question}
Does there exist an integer $N$ such that if $G$ is a tree on $n$ vertices, then
$$\max_{\mathcal{A}}CC_G(\mathcal{A})\leq\frac{1}{4}n+N\;?$$
\end{question}

Let $\mathcal{F}_k$ be a family of all maximal $k$-degenerate graphs. By Theorem \ref{kTheoremMain} and Theorem \ref{kTheoremBlind} we have
$$\limsup_{G\in\mathcal{F}_k,\vert G\vert\rightarrow\infty}\frac{\max_{\mathcal{A}}CC_G(\mathcal{A})}{\vert G\vert}=\frac{k^k}{(k+1)^{k+1}}.$$	
For $k$-trees we have equality $\limsup=\liminf$, but it is not hard to show a family of maximal $k$-degenerate graphs for which $\liminf$ is smaller.

\begin{question}
Determine the value
$$\liminf_{G\in\mathcal{F}_k,\vert G\vert\rightarrow\infty}\frac{\max_{\mathcal{A}}CC_G(\mathcal{A})}{\vert G\vert}.$$	
\end{question}

Another direction is to consider different classes of graphs. In our opinion it is natural to examine lattice graphs, ex. $d$-dimensional grids -- they are almost maximal $d$-degenerate, but when $d\geq 2$ their tree-width is unbounded, so they are not subgraphs of $k$-trees (for any $k$).

\begin{question}
What is the maximum expected number of connected components a stopping algorithm can guarantee on a $d$-dimensional grid?
\end{question}

It is natural to examine triangulated planar graphs -- they include $2$-dimensional grids, and are $5$-degenerate.

\begin{question}
What is the maximum expected number of connected components a stopping algorithm can guarantee on a triangulated planar graph?
\end{question}

We can also change the notion of connectivity and receive a quite different problem.

\begin{question}
Suppose edges of an $n$-clique $K_n$ become active on-line in time, one by another, in a random order. Find a stopping algorithm that maximizes the expected number of $2$-connected components in the active part.  Find this maximum expected number.
\end{question}

\section{Acknowledgements}

We thank Ma{\l}gorzata Sulkowska for stimulating discussions and helpful comments.
	

\end{document}